\documentclass[12pt]{amsart}
\usepackage{amsmath,amssymb,amsfonts, amscd}
\usepackage{amsbsy}
\usepackage{amsthm}
\usepackage{amstext}
\usepackage{amsopn}
\usepackage{hyperref}
\usepackage{mathrsfs}
\usepackage{graphicx}
\usepackage{latexsym}
\usepackage[T1]{fontenc}
\usepackage{color}

\newcommand{\Hmm}[1]{\leavevmode{\marginpar{\tiny%
$\hbox to 0mm{\hspace*{-0.5mm}$\leftarrow$\hss}%
\vcenter{\vrule depth 0.1mm height 0.1mm width \the\marginparwidth}%
\hbox to 0mm{\hss$\rightarrow$\hspace*{-0.5mm}}$\\\relax\raggedright #1}}}

\newtheorem{thm}{Theorem}[section]
\newtheorem{cor}[thm]{Corollary}

\newtheorem{lemma}[thm]{Lemma}
\newtheorem{pro}[thm]{Proposition}

\theoremstyle{definition}

\newtheorem{eg}[thm]{Example}

\newcommand{\Z}{{\mathbb Z}}
\newcommand{\R}{{\mathbb R}}
\newcommand{\C}{{\mathbb C}}
\newcommand{\Q}{{\mathbb Q}}

\newcommand{\al}{{\alpha}}

\newcommand{\eps}{{\varepsilon}}

\newcommand{\ka}{{\Phi}}
\newcommand{\si}{{\sigma}}
\newcommand{\lm}{{\lambda}}
\newcommand{\ph}{{\varphi}}

\newcommand{\ov}[1]{\overline{ #1}}
\newcommand{\ow}[1]{\widetilde{ #1}}

\begin{document}
\title[Geometry and spectrum of tessellations]{Geometric and spectral consequences of curvature bounds on tessellations}

\author[M. Keller]{Matthias Keller}
\address{Universit\"at Potsdam, Institut f\"ur Mathematik, 14476  Potsdam, Germany} \email{matthias.keller@uni-potsdam.de}


\date{\today}
\maketitle


\section{Introduction}

This chapter focuses on the theme of geometric and spectral consequences of curvature bounds. The geometric setting are tessellations of surfaces with finite and zero genus.  Our main focus is  on infinite tessellations. Several of the results  presented here have analogues in Riemannian geometry. In some cases one can go even beyond the Riemannian results and there also striking differences which shall be highlighted.

The curvature under investigation arises as a combinatorial quantity with the geometric interpretation of an angular defect.
For a vertex $v$ the curvature is given by
\begin{align*}
\ka(v)=1 - \frac{d_{v}}{2}+\sum_{f \mbox{\scriptsize{ face with }}v\in f}\frac{1}{\deg(f)}= \frac {1}{2\pi}\Big(2\pi-\sum_{f \mbox{\scriptsize{ face with }}v\in f}\beta(f)\Big),
\end{align*}
where $d_{v}$ is the vertex degree, $\deg(f)$ the number of vertices contained in a face $f$ and $\beta(f)=(\deg(f)-2)/2\deg(f)$ the inner angle of $f$ considered as a regular $\deg(f)$-gon. We introduce this concept in detail and discuss the basic features below which is the first part of the chapter.

In the second part of the chapter, in Section~\ref{s:geometry}, geometric consequences of curvature bounds are discussed. Here, the discrete Gau\ss-Bonnet Theorem provides a starting point from which various directions shall be explored. First, it is discussed that positivity of the curvature yields finiteness of the graph which is an analogue to a theorem of Myers from Riemannian geometry. In contrast, non-positive curvature yields infinite tessellations for which we study further geometric properties. Specifically, a Hadamard-Cartan theorem  is  presented which states that geodesics can be continued indefinitely as a consequence of non-positive curvature. Furthermore,  volume growth bounds are derived from upper and lower curvature bounds.  Finally, hyperbolicity properties such as strong isoperimetric inequalities and Gromov hyperbolicity as consequences of negative curvature are studied.

In the third part of the chapter, in Section~\ref{s:spectrum}, we focus on spectral properties of the combinatorial Laplacian
\begin{align*}
    \Delta\ph(v)=\sum_{w\sim v}(\ph(w)-\ph(v))
\end{align*}
on $\ell^{2}(V)$.  Let us stress that this combinatorial Laplacian  differs from the normalized Laplacian introduced in the previous chapter (which is also referred to as the Tutte  Laplacian or harmonic Laplacian). The main reason is that the operator above captures phenomena related to unbounded geometry for infinite graphs much better.

Many of the spectral properties of $\Delta$ are consequences of the geometric properties which were established before. First, we discuss whether the combinatorial Laplacian admits a spectral gap due to  the curvature being non-negative or negative which includes an analogue to a classical theorem of McKean.
Secondly, we focus on the case of uniformly decreasing curvature which was studied in the Riemannian setting  by Donnelly/Li. Here, we not only characterize the situation, where the combinatorial Laplacian has only discrete eigenvalues as spectrum but also determine the asymptotics of eigenvalues as well as the decay of eigenfunctions. Here, it appears that the eigenvalue asymptotics behave differently than in the case of manifolds.
Thirdly, we explore the phenomena of compactly supported eigenfunctions which is often studied under the term unique continuation of eigenfunctions. It came as a surprise  that in contrast to the continuous setting such compactly supported eigenfunctions appear in the discrete setting. However, they can be excluded in situations of non-positive curvature.
Finally, we briefly consider the $\ell^{p}$-spectrum of the combinatorial Laplacian and ask when it differs from the $\ell^{2}$ spectrum, which are discrete analogues of results of Sturm.

In the outlook section, which forms the fourth part, we discuss how the results of the first  parts can be generalized to  planar graphs or more general polygonal complexes including two dimensional buildings.

All the results are discussed in the light of the present literature which reaches back from the seventies to the very present and has contributions from various schools of researchers. We  sketch the key ideas of the proofs without going into too much detail.

\subsection{Graphs}
Let a connected simple graph  $(V,E)$ be given. We call two vertices $x,y\in V$ which are connected by an edge \emph{adjacent} and denote $x\sim y$.  We call a sequence $(x_{r})_{r\ge0}$ of subsequently adjacent and pairwise distinct vertices a \emph{path}. For a finite path  $(x_{0},\ldots,x_{n})$, we call $n$ the \emph{length} of the path.    For two vertices $x,y\in V$, we define the \emph{distance} $d(x,y)$ of $x$ and $y$ to be the  length of the shortest path connecting $x$ and $y$. We call a path $(x_{r})$ a geodesic if $d(x_{0},x_{n})=n$ for all indices $n$ of the path.

The \emph{sphere} $S_{r}=S_{r}(o)$ of radius $r$ about a fixed vertex $o$ (which is mostly suppressed in notation) includes the vertices with distance $r$ to $o$. The \emph{ball} $B_{r}=B_{r}(o)$ of radius $r$ about $o$ is the union of all spheres of distance less or equal to $r$.

\subsection{Tessellations of  surfaces}
In this chapter we consider tessellations. From a purely graph-theoretic approach this might seem rather restrictive but it is  natural from a discretization and geometric point of view and it allows to prove finer results.

We assume that $(V,E)$ can be \emph{embedded} in an oriented topological surface $\mathcal{S}$ {without self intersections}. That means that the vertices $V$ are identified with points in $\mathcal{S}$ and the edges $E$ are identified with simple curves connecting its end vertices such that two different edges intersect only in the vertices they have in common. Throughout this chapter we will not distinguish between the graph and its embedding.

We always assume that the embedding is \emph{locally compact}, i.e.,  every compact $K\subseteq \mathcal{S}$ intersects only finitely many edges.

A graph is called \emph{planar} if there is a locally compact embedding into a surface $\mathcal{S}$ homeomorphic to $\R^{2}$ and \emph{spherical} if $\mathcal{S}$ can be chosen to be homeomorphic to the  sphere $\mathbb{S}^{2}$.

For an embedded graph $(V,E)$, we let the set of faces $F$ consist of  the closures of the connected components of
\begin{align*}
    \mathcal{S}\setminus \bigcup E.
\end{align*}
We denote $G=(V,E,F)$ and following \cite{BP1,BP2} we call $G$ a \emph{tessellation} or a \emph{tiling} if
\begin{itemize}
  \item [(T1)] every edge is included in   two faces,
  \item [(T2)] every two faces are either disjoint or intersect in one vertex or one edge,
  \item [(T3)] every face is homeomorphic to a closed disc.
\end{itemize}

Let us shortly discuss the local compactness assumption in the light of tessellations. This assumption can be violated for two reasons. The first is that the point where the assumption fails is a vertex. This, however, implies that infinitely many edges emanate from this vertex. So, a locally compact embedding first of all implies that the graph is locally finite, i.e., each vertex has only finitely many neighbors. Secondly, the assumption can fail for a point which is not a vertex. Removing all such points from the surface, annihilates the problem and the embedding becomes locally compact. So, in this case the graph can be embedded  locally compactly in a different surface.


\subsection{Curvature}
To define a curvature function, we introduce the vertex degree and the face degree. First we let  the
\emph{degree} of a vertex $v\in V$ be defined as
\begin{eqnarray*}
    d_{v}=\#\mbox{edges emanating from $v$}.
\end{eqnarray*}
The \emph{face degree} of a face $f\in F$ is defined as
\begin{eqnarray*}
    \deg(f)=\#\mbox{boundary edges of $f$}    =\#\mbox{boundary vertices of $f$}.
\end{eqnarray*}
The \emph{vertex curvature} $\ka:V\to \R$ is defined as
\begin{eqnarray*}
\ka(v)=1-  \frac{d_{v}}{2}+\sum_{f\in F, v\in {f}}\frac{1}{\deg(f)}.
\end{eqnarray*}
This idea is already found in the works of  Descartes, see e.g. \cite{Fed}. Later it was
introduced in the above form by Stone, \cite{S1}, where he refers ideas going back to Alexandrov. The notion reappeared in various settings \cite{Gr, Ish} and was studied intensively since then \cite{BP1,BP2,DM,Hi,HJL,K1,K2,KP, Oh,Woe,Zuk}.

This  notion of curvature is motivated by an angular defect as follows: Think of  a face $f$ as a regular polygon. Then, the interior angles of $f$ are equal to
\begin{align*}
    \beta(f)=2\pi\frac{\deg(f)-2}{2\deg(f)}.
\end{align*}
This formula is easily derived as going around $f$ results in an angle of $2\pi$, while going around the $\deg(f)$ corners of $f$ one takes a turn by an angle of $\pi-\beta(f)$ each time. With this interpretation the  curvature of  a vertex $v\in V$ can be rewritten as
\begin{align*}
    2\pi\ka(v)= 2\pi-\sum_{f\in F, v\in f}\beta(f).
\end{align*}
Despite of this interpretation, it should be stressed that the mathematical nature of $\ka$ is purely combinatorial. However,  when we assume an embedding such that every polygon is regular we have the geometric interpretation above. Furthermore, the curvature satisfies a
Gau{\ss}-Bonnet formula which is found in the next section in  Theorem~\ref{t:GaussBonnet}.

Before we come to this, we introduce a finer notion of curvature.  To this end, we let the set of \emph{corners} $C(G)$ of the graph $G$ be given by the set of pairs $(v,f)\in V\times F$ such that $v$ is contained in $f$.  One can picture the corners around a vertex as the connected components of a small neighborhood of $v$ in $\mathcal{S}$ after removing the edges emanating from $v$. The \emph{corner curvature} $\ka_{C}:C(G)\to \R$ is given by
\begin{align*}
    \ka_{C}(v,f) =\frac{1}{d_{v}}-\frac{1}{2} +\frac{1}{\deg(f)}.
\end{align*}
This quantity was first introduced in \cite{BP1,BP2} for tessellations  and in \cite{K2} for general planar graphs (cf. Section~\ref{s:planar}). The corner curvature measures the contribution of every corner of a vertex to the vertex curvature, i.e., summing over all corners of a  vertex, we get
\begin{align*}
    \ka(v)=\sum_{(v,f)\in C(G)}\ka_{C}(v,f).
\end{align*}
With the interpretation of the angular defect above the corner curvature can be expressed as
\begin{align*}
     2\pi\ka_{C}(v,f)= \frac{1}{d_{v}}(2\pi-d_{v}\beta(f)),\quad (v,f)\in C(G).
\end{align*}

We say a tessellation is \emph{flat} if $\ka\equiv 0$. Let us discuss a few examples.

\begin{eg}We say a tessellation of $\mathcal{S}$  is $(p,q)$-regular if $q$ regular polygons of degree $p$ meet in every vertex. Then, the sign of vertex curvature is already determined by the sign of the corner curvature which is constantly $1/q-1/2+1/p$.
Let us come to more specific examples.

(a) The examples of $(p,q)$-regular tessellations of $\mathbb{S}^{2}$ of positive curvature are the platonic solids, i.e., the tetrahedron $(3,3)$, the octahedron $(3,4)$, the icosahedron $(3,5)$, the cube $(4,4)$ and the dodecahedron $(5,3)$. 

(b) Flat $(p,q)$-regular tessellations of $\R^{2}$ are exactly either the square tessellation associated to $\Z^{2}$ of type $(4,4)$, the flat triangular tessellation $(3,6)$ and the hexagonal honeycomb tessellation $(6,3)$.
Furthermore, any large enough rectangular subset of $\Z^{2}$ can be realized as a tessellation of the torus in the usual way by identifying opposite sides. This  results again flat curvature. Similar constructions can be realized for the other two flat tessellations.

(c) Planar $(p,q)$-regular tessellations of the type $(4,5)$, $(5,5)$, $(5,6)$, $(5,4)$, $(6,5)$ or such that $p\ge7$ or $q\ge3$ give rise to tessellations of the hyperbolic plane with negative curvature.

See Figure~\ref{f:curvature} for examples of (a), (b) and (c).
\end{eg}
\newpage

\begin{figure}[!h]
\scalebox{0.17}{\includegraphics{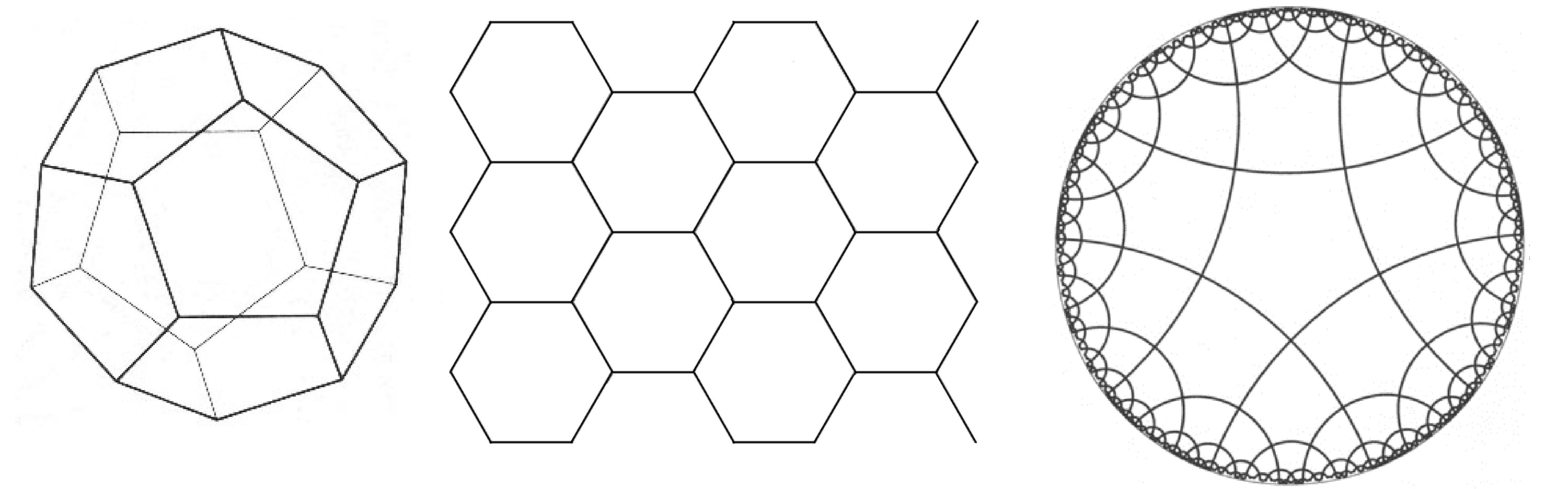}}
\caption{\label{f:tilings}(a) The dodecahedron, (b) the hexagonal tiling and (c) the $(5,4)$ tessellation embedded in the Poincar\'e disc.}\label{f:curvature}
\end{figure}

Next, we discuss  examples of non definite curvature, see Figure~\ref{f:Kagome} as well.

\begin{eg}The so called Cairo tiling consists of pentagons. For each pentagon there are two non adjacent vertices that are contained in four other pentagons. This results in negative corner and vertex curvature. The other three vertices are contained in three pentagons resulting in positive corner and vertex curvature.
\end{eg}

\begin{eg}The Penrose tiling consists of squares. There are vertices adjacent to five and six squares resulting in negative curvature, vertices adjacent to four squares resulting in zero curvature and vertices adjacent to three squares resulting in positive curvature. This holds for the corner curvatures as well as for the vertex curvature.
\end{eg}

Finally, we discuss examples of non-positive and negative vertex curvature which have yet non definite corner curvature.

\begin{eg}\label{e:Kagome} The class of example we discuss next consists of regular polygons of degree $p\ge6$ and of triangles. In each vertex two triangles and two $p$-gons meet such that faces of the same type are opposite to each other. The vertex curvature in a vertex is then $1/p-1/6\leq0$. However, the curvature in corners of a triangle is $1/12$ and in the corners of a $p$-gon is $1/p-1/2<0$. Hence, the corner curvature is non-definite. A special case with $q=6$ is  called the  trihexagonal tiling or the Kagome lattice or the hexadeltille.
\end{eg}

\begin{figure}[!h]
\scalebox{0.3}{\includegraphics{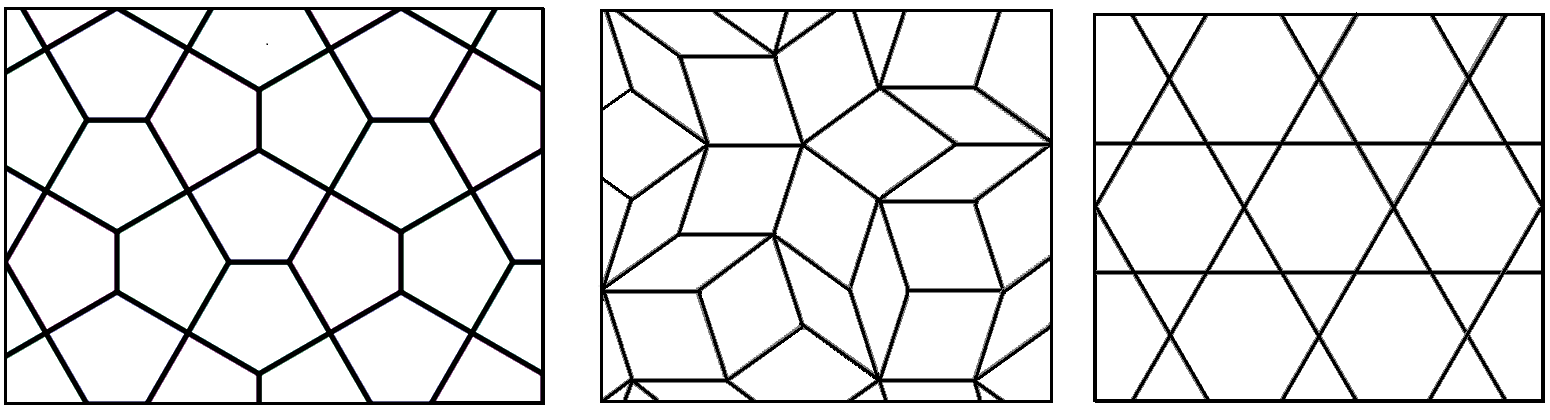}}
\caption{\label{f:tilings}The Cairo tiling, the Penrose tiling and
the Kagome lattice.}\label{f:Kagome}
\end{figure}

\subsection{Duality}
We make some brief remarks about duality. To every tessellation $G=(V,E,F)$ we can relate a dual tessellation $G^{*}=(V^{*},E^{*},F^{*})$ in the following way: To define the vertex set we chose for each face in $F$ a point in its interior which we refer to as a vertex in $V^{*}$. If two faces in $F$ intersect in an edge we connect the corresponding vertices in $V^{*}$ by an edge. This gives a graph embedded in the same surface and it can be seen that the faces of $(V^{*},E^{*})$ can be related one-to-one to the vertices in $V$. Let $C(G^{*})$ be the corners of $G^{*}$. We see that, there are bijective maps from $V$ to $F^{*}$, $E$ to $E^{*}$, $F$ to $V^{*}$ and $C(G)$ to $C(G^{*})$. Denote the bijective maps $V\to F^{*}$, $v\mapsto {v}^{*}$ and $F\to V^{*}$, $f\mapsto {f}^{*}$ and observe that
\begin{align*}
    d_{v}=d_{v^{*}}\quad\mbox{and}\quad\deg(f)=|f^{*}|.
\end{align*}
We stress that $d_{v^{*}}$ is a face and $f^{*}$ is a vertex in $G^{*}$.
We denote the corner curvature  of $G^{*}$ by $\ka^{*}_{C}$ and observe that
\begin{align*}
    \ka_{C}(v,f)=\frac{1}{d_{v}}-\frac{1}{2}+\frac{1}{\deg(f)} =\frac{1}{_{v^{*}}}-\frac{1}{2}+\frac{1}{|f^{*}|}=\ka_{C}^{*}(f^{*},v^{*}).
\end{align*}
Moreover, defining the face curvature $\ka_{F}^{*}:F^{*}\to\R$ on $G^{*}$ viz
\begin{align*}
    \ka_{F}^{*}(f)=\sum_{(v, f)\in C(G^{*})}\ka_{C}^{*}(v,f)= 1-\frac{\deg(f)}{2}+\sum_{v\in V^{*}, v\in{f}}\frac{1}{d_{v}},
\end{align*}
for $ f\in F^{*}$,
we see that
\begin{align*}
    \ka(v)=\sum_{(v, f)\in C(G)}\ka_{C}(v,f)=\sum_{(f^{*},v^{*})\in C^{*}(G)}\ka_{C}^{*}(f^{*},v^{*})=\ka_{F}^{*}(v^{*}),
\end{align*}
for $v\in V$.
Similarly, if we define a face curvature on $G$, then the vertex curvature on $G^{*}$ can be related analogously.
We summarize that for each tessellation $G$ the dual graph $G^{*}$ gives rise to a tessellation. Moreover, these two graphs have the same corner curvature with respect to the canonical bijection. The vertex curvature in $G$ translate into a face curvature in $G^{*}$ and vice versa.

\section{Geometry}\label{s:geometry}
\subsection{Gau\ss-Bonnet Theorem}

In Riemannian geometry the Gau\ss-Bonnet Theorem is a link between the geometry  and topology of a surface. In particular, it relates the curvature of a surface to its Euler characteristic. An analogous theorem holds true in the discrete setting.
Various versions of the theorem below are found in \cite{BP1,Chen,CC,DM,K2}.

For a surface $\mathcal{S}$ the genus $g$ is defined to be the largest number of nonintersecting simple closed curves that can be drawn on the surface without separating it. It can be thought as the number of holes in a surface. The Euler characteristic $\chi(\mathcal{S})$ of $\mathcal{S}$ is given as
\begin{align*}
    \chi(\mathcal{S})=2-2g.
\end{align*}

\begin{thm}[Gau\ss-Bonnet Theorem]\label{t:GaussBonnet} Let $G=(V,E,F)$ be a tessellation embedded locally compactly in a compact oriented surface $\mathcal{S}$ of finite genus. Then,
\begin{align*}
    \sum_{v\in V}\ka(v)=\chi(\mathcal{S}).
\end{align*}
\end{thm}
\begin{proof}[Sketch of the proof] Since the embedding is locally compact, the tessellation $G$ embedded in a compact surface is finite. We calculate
\begin{align*}
    \sum_{v\in V}\ka(v)=&  \sum_{v\in V}\sum_{(v,f)\in C(G)}\Big(\frac{1}{d_{v}}-\frac{1}{2}+\frac{1}{\deg(f)}\Big).
\end{align*}
Using  $\#\{(v,f)\in C(G)\}=d_{v}$ for  fixed $v\in V$, we arrive at
\begin{align*}
    \ldots=&|V|-\frac{1}{2}\sum_{v\in V}d_{v}+\sum_{f\in F}\sum_{(v,f)\in C(G)}\frac{1}{\deg(f)}.
\end{align*}
Now, employing $\sum_{v\in V}d_{v}=2|E|$ and $\#\{(v,f)\in C(G)\}=\deg(f)$ for  fixed $f\in F$, yields
\begin{align*}
    \ldots=&|V|-|E|+|F|.
\end{align*}
The statement now follows from the identity $|V|-|E|+|F|=\chi(\mathcal{S})$ which is known as Poincar\'e formula. This formula is proven by induction over $g$. The base case $g=0$ is known as the Euler-Descartes formula. The induction step for a tessellation $G_{g}=(V_{g},E_{g},F_{g})$ embedded in a  surface $\mathcal{S}_{g}$ of genus $g$ works by cutting along a path of edges that separates the surface. Say the length of the path is $p$. The cut surface has now a boundary consisting of two paths of edges of length $p$. In each of these boundary paths one glues a  polygon of degree $p$ to  ``close up'' the surface in order to get a surface of genus $g-1$. For this surface we know the formula by the induction hypothesis. Now by cutting and gluing, we increased the number of vertices by $p$, we also increased the number of edges also by $p$ and we increased the number of faces by $2$. This results in the formula $|V_{g}|-|E_{g}|+|F_{g}|=\chi\mathcal({S}_{g})$.
\end{proof}

\subsection{Approximating flat and infinite curvature}

From the definition of the curvature it follows that $\ka$ takes only values in $\Q$. So a question is which values can actually be obtained or approximated. Here, we discuss that $\ka$ is bounded from above but not from below. Moreover, we find that flat curvature can be approximated from above but not from below.

Clearly, the minimal vertex and face degree is $3$ by assumption. Hence, the maximal value that can be assumed by $\ka$ is $3/2$.

On the other hand, $\ka$ is not necessarily bounded from below as it can be seen by the next proposition.
\begin{pro}\label{p:ka_vs_deg} Let $G=(V,E,F)$ be a tessellation. Then, for all $v\in V$
\begin{align*}
-\frac{d_{v}}{2} \leq \ka(v)\leq 1-\frac{d_{v}}{6}.
\end{align*}
In particular, $\ka(v)\leq 0$ if $d_{v}\ge 6$ and
for $(v_{n})$ we have $d_{v_n}\to\infty$ if and only if $\ka(v_{n})\to-\infty$ for $n\to\infty$.
\end{pro}
\begin{proof} The lower bound is immediate and the upper bound follows as $\deg(f)\ge3$ for all $f\in F$ and the number of corners about a vertex equals the vertex degree.
\end{proof}

Let us now turn to flat curvature. We discuss that the value zero can be approximated from below but not from above.

For positive curvature we consider so called prisms or antiprisms, see Figure~\ref{f:prism}. In the figure, the unbounded face outside arises from a face in the sphere.

A \emph{prism} is a tessellation of the sphere. It is given by two polygons of degree $p$ that are surrounded by $p$ squares. In particular, each vertex is adjacent to two squares and one of the two $p$-gons. The curvature  of any vertex $v$ is given by
\begin{align*}
    \ka(v)=1-\frac{3}{2}+\Big(\frac{1}{p}+2\frac{1}{4}\Big) =\frac{1}{p}.
\end{align*}

Similarly, an \emph{antiprism} consists of two $p$-gons that have triangles glued along the boundary such that in the resulting tessellation  every vertex is adjacent to three triangles and one $p$-gon.
 The curvature of any vertex  $v$ is given by
\begin{align*}
    \ka(v)=1-\frac{4}{2}+\Big(\frac{1}{p}+3\frac{1}{3}\Big)=\frac{1}{p}.
\end{align*}
\begin{figure}[!h]
\scalebox{0.4}{\includegraphics{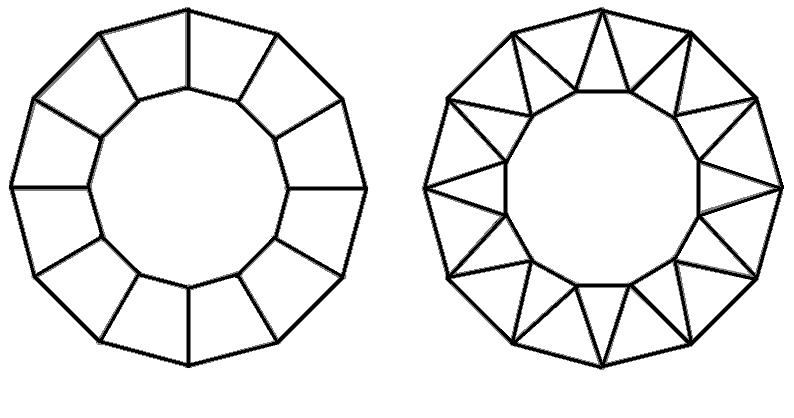}}
\caption{\label{f:tilings}A prism and an antiprism projected from the sphere to the Euclidian plane.}\label{f:prism}
\end{figure}

In the case when $p$ tends to $\infty$, the curvature of prisms and antiprisms tends to $0$ from above.

In the case of negative curvature there is  a uniform lower bound which was shown by Higuchi, \cite{Hi}, via listing all essential cases.

\begin{thm}\label{t:Higuchi}\cite[Proposition~2.1.]{Hi} Let $G$ be a tessellation. If $\ka<0$, then $\ka\le -1/1806$. This maximum is can be assumed at a vertex $v$ with $d_{v}\ge3$ and which is adjacent to faces $f_{1},f_{2}, f_{3}$ such that $(|f_{1}|,|f_{2}|, |f_{3}|)=(3,7,43)$.
\end{thm}

The theorem implies that flat curvature can not be approximated from below.

\subsection{Finiteness}

Next, we turn to question how the sign of the curvature determines finiteness or infiniteness of the tessellation. A rough synopsis is that positive  curvature implies finiteness and in the case of planar tessellations non-positive curvature implies infiniteness.

In Riemannian geometry a theorem that positive curvature implies compactness of the manifold goes back to Myers, \cite{Myers}. For graphs this question was first studied in \cite{S1,S2} and came up again later as Higuchi's conjecture \cite[Conjecture~3.2.]{Hi}. A definite answer was given by DeVos/Mohar \cite[Theorem 1.7]{DM} after first steps were taken in \cite{CC,SY}.

\begin{thm}[Myer's theorem for tessellations, \cite{DM}] Let $G=(V,E,F)$ be a tessellation embedded locally compactly in a surface such that $\ka>0$. Then, the graph is finite.
\end{thm}
The arguments of the proof use an asymptotic Gau\ss-Bonnet formula for infinite graphs with non-negative curvature. Furthermore, it is then shown that \emph{big faces}  of degree larger than 42 can not be to close. In particular, if an edge connects two big faces by its end vertices then the graph is a prism or antiprism which  is  essentially proven by listing all possible cases. Then, a lower bound on the sum of the curvature is derived in terms of the number of vertices for graphs which are neither prisms or antiprisms.

The upper bound for the number of vertices of a graph with positive curvature given in \cite{DM} for graphs that are not prisms or antiprisms is $3444$. This was later improved by  Zhang \cite{Zh} to 580 vertices while  the largest known graphs with positive curvature has 208 vertices    and was constructed by Nicholson and Sneddon \cite{NS}.

Next, we turn to results which complement the result above. An interpretation of the theorem below is that non-positive curvature implies that a planar tessellation is infinite.

\begin{thm} Let $G=(V,E,F)$ be a tessellation embedded locally compactly in $\mathbb{S}^{2}$. Then, the graph admits some positive curvature.
\end{thm}
\begin{proof} The statement follows directly from the Gau\ss-Bonnet theorem.
\end{proof}

We end this section by  a result of Chen \cite{Chen} that states that an infinite tessellation with non-negative curvature can have at most finitely many vertices with positive curvature.

\begin{thm}\cite[Theorem 3.5.]{Chen} Let $G$ be a tessellation  such that $\ka\ge0$. Then,  the number of vertices with non-vanishing curvature is finite.
\end{thm}

\subsection{Absence of cut locus}
In Riemannian geometry the Hadamard-Cartan theorem states that on the universal cover of a complete manifold with non-positive sectional curvature all geodesics can be continued indefinitely.

For graphs such a theorem can be proven under the assumption of non-positive corner curvature. Below we will discuss examples how such a statement may fail for non-positive (or even negative) vertex curvature.

For a vertex $v$ the \emph{cut locus} is the set of vertices, where the  distance function $d(v,\cdot)$ assumes a local maximum. We say the graph has \emph{empty cut locus} if the cut locus of every vertex is empty. This is equivalent to the fact that one can continue every finite geodesic indefinitely.

\begin{thm}\label{t:absence_cut_locus}\cite[Theorem~1]{BP2} Let $G$ be a planar tessellation such that $\ka_{C}\le0$. Then, the graph has empty cut locus.
\end{thm}

The proof of the theorem involves an analysis of the boundary structure of distance balls of tessellations with non-positive corner curvature. In \cite{BP2} it is shown that the balls satisfy a subtle notion of convexity which is referred to as admissibility.

The statement of the theorem can easily be seen to fail if one only assumes $\ka\leq0$. An  example is the Kagome lattice of Example~\ref{e:Kagome}. This tessellation satisfies $\ka \leq 0$ but has corners of positive curvature, so, it does not satisfy the assumption of Theorem~\ref{t:absence_cut_locus}. Indeed, for an arbitrary vertex $o$ there are two vertices in $B_{3}$ that have only neighbors in $B_{3}$ -- these are the opposite vertices in the hexagons that contain $o$.

For negative vertex curvature, consider a tessellation with $2p$-gons, $p\ge4$, instead of hexagons as also mentioned in  Example~\ref{e:Kagome}. Then this tessellation satisfies $\ka<0$ but also has corners of positive curvature. Indeed, every vertex $o$ has vertices in $B_{p}$ that have no neighbor in $S_{p+1}$.

\subsection{Volume growth}\label{s:volume}

We turn to implications of the sign of the curvature on the volume growth of planar tessellations. The rough synopsis of the exposition below is that non-negative curvature implies polynomial growth and for negative curvature implies at least exponential growth.

In \cite{HJL} Hua/Jost/Liu showed that a tessellation of non-positive curvature does not grow faster than quadratically.

\begin{thm}\label{t:grow_nonneg_curvature}\cite[Theorem~1.1]{HJL} Let $G$ be a tessellation embedded locally compactly in a surface of finite genus such that $\ka\ge0$. Then, there is a constant $C$ depending only on the maximal face degree such that for all $r\ge0$
\begin{align*}
    \# B_{r}\leq C r^{2}.
\end{align*}
\end{thm}
The strategy of the proof is to relate volume growth of the tessellation to the one of the corresponding Alexandrov space in which it is canonically embedded.

Next, we turn to the case of negative curvature.
We first discuss lower bounds on the volume growth that are proven in \cite{BP1}.

\begin{thm}\cite[Corollary 5.2.]{BP1} Let $G=(V,E,F) $ be a planar tessellation that has empty cut locus and satisfies $\ka\leq -k<~0$. Then,
$$\#B_{r}\geq (1+2Ck)^{r},$$
where $C=p/(p-1)$ if $p=\sup_{f\in F}\deg(f)<\infty$ and $C=1$ otherwise. 
\end{thm}

The assumption of emptiness of cut locus is for example implied by non-positive corner curvature, Theorem~\ref{t:absence_cut_locus}.

Next, we turn to upper bounds. There is a rough immediate result by comparing the tessellation with a tree. Note that for a regular trees with vertex degree $q$, the number of vertices in $S_{r}$ is exactly $(q-1)^{r}$. We estimate the volume growth of a tessellation by removing edges within spheres and splitting geodesics from $o$ that intersect in a vertex into two different ones. This results in the theorem below.

\begin{thm}\cite[Theorem~4]{KP} Let $G=(V,E,F) $ be a planar tessellation such that $q=\sup_{v\in V}d_{v}<\infty$. Then,
$$\#B_{r}\leq (q-1)^{r}.$$
\end{thm}

For tessellations whose faces  have  fixed degree $p$, one can give an explicit recursion formula for the size of the distances spheres.  The recursion is expressed in terms of  \emph{normalized average curvatures} over
spheres $$\overline\Phi_r:=\overline{\Phi}(S_r) := \left({\frac{2p}{p-2}}\right)\frac{1}{\#S_r}\sum_{v\in S_{r}}\ka(v).$$
Recall that the constant $2\pi(p-2)/2p$ is the internal angle $\beta(f)$ of a regular $p$-gon $f$.

Similar to the lower bound above the growth formula can be proven for tessellations without \emph{cut locus}  which is implied by non-positive corner curvature, Theorem~\ref{t:absence_cut_locus}. In the case of face regular graphs non-positive corner curvature is equivalent to non-positive curvature. However, the theorem below is not restricted to the non-positive curvature case.

For $3 \le p < \infty$, let $N = \frac{p-2}{2}$ if
$p$ is even and $N = p-2$ if $p$ is odd, and
\begin{equation*} \label{eq:bl}
  b_l = \begin{cases} \frac{4}{p-2}-2 &:\; \text{if $p$ is odd and $l =
    \frac{N-1}{2}$,}\\
    \frac{4}{p-2} &:\; \text{else,} \end{cases}
\end{equation*}
for $0 \le l \le N-1$.

\begin{thm}\label{thm:facereggrow} \cite[Theorem~2]{KP}
  Let $G=(V,E,F)$ be a $p$-face regular  tessellation with empty cut locus.
  Then  the following   $(N+1)$-step recursion formulas for $r \ge 1$ holds
  \begin{equation*} \label{eq:sigman}
\#  S_{r+1}
  = \begin{cases}
  \sum_{l=0}^{r-1} (b_l -  \overline{\Phi}_{r-l})
    \#S_{r-l} + \#S_1&:\; \text{if $r < N$,} \\[.3cm]
    \sum_{l=0}^{N-1} (b_l - \overline{\Phi}_{N-l}) \#S_{N-l} &:\; \text{if $r =
      N$,} \\[.3cm]
     \sum_{l=0}^{N-1} (b_l - \overline{\Phi}_{r-l}) \#S_{r-l}  - \#S_{r-N}&:\;
    \text{if $r > N$.} \end{cases}
  \end{equation*}
\end{thm}

The $(N+1)$-step recursion formula gives rise to a \emph{recursion matrix} $M_r$, $r\ge0$, mapping $\R^{N}$ to $\R^{N}$ such that $M_{r}(\#S_{r-N},\ldots, \#S_{r})=(\#S_{r-N+1},\ldots, \#S_{r+1})$.

In the special case when also the vertex degree is constant, say $q$, we have a $(p,q)$-regular tessellation. Then, the constant $b_{l}-\ov\ka_{k}$ is equal to $q-2$, except for $l=(N-1)/2$ and $q$ odd in which case we have $q-4$.
In particular, there is a matrix $M$ such that $M=M_{r}$ for all $r\ge0$. The characteristic polynomial of $M$ is then given by the complex polynomial
\begin{eqnarray*}
  g_{p,q}(z) = 1 - (q-2) z - \dots - (q-2) z^N + z^{N+1},
\end{eqnarray*}
if $p$ is even, and
\begin{eqnarray*}
  g_{p,q}(z) = 1 - (q-2)z - \dots - (q-4)z^{\frac{N+1}{2}} - \dots - (q-2)z^N
  + z^{N+1},
\end{eqnarray*}
if $p$ is odd.  By \cite{CaWa} and \cite{BaCS}, $g_{p,q}$ is a reciprocal Salem polynomial, i.e., its roots lie on the
complex unit circle except for two positive reciprocal real zeros
$$\frac{1}{x_{p,q}} < 1 < x_{p,q} < p-1.$$
This yields
\begin{align*}
    \limsup_{t\to\infty}\frac{1}{r}\log\#S_{r}=\log x_{p,q}
\end{align*}
in the special case of a $(p,q)$-regular tessellation.
In particular, the considerations above recover the results of Cannon and Wagreich \cite{CaWa} and Floyd and Plotnick \cite[Section 3]{FP} that the growth function  is a rational function.

To relate growth rate of an arbitrary face regular tessellation to a $(p,q)$-regular tessellation we present a volume growth comparison theorem which is an analogue to the Bishop-Guenther-Gromov comparison theorem from Riemannian geometry.

\begin{thm}[Theorem~3 in \cite{KP}] \label{t:growest}
  Let $G = (V,E,F)$ and $\widetilde G = (\widetilde V,\widetilde   E, \widetilde F)$ be two $p$-face regular tessellations with   non-positive vertex curvature and let $S_r \subset V$ and $\widetilde S_r   \subset \ow V$ be spheres with respect to the centers $o \in V$ and
  $\widetilde o \in \widetilde V$, respectively. Assume that the normalized   average spherical curvatures satisfy   $$ \overline{\Phi}(\widetilde S_r) \le \overline{\Phi}( S_r) \le 0, \quad
 \ r\ge 0. $$
  Then the difference sequence $(\#\widetilde S_r - \# S_r)$ satisfies $ \#\widetilde S_r - \# S_r\ge 0$, $r\ge0$, and
  is monotone non-decreasing.
\end{thm}

\begin{proof}[Idea of proof] The proof given in \cite{KP} depends heavily on the assumption of constant face degree. Let us briefly illustrate the underlying idea. We count the number of faces $c_{j}^{r}$ that intersect a ball $B_{r}$ in $1\leq j\leq p$ vertices, where $p$ is the constant face degree. If $j\leq p-2$, then this number equals the number of faces $c_{j+2}^{r+1}$ that intersect the ball $B_{r+1}$ in $j+2$ vertices. Finally, one can relate the numbers $c_{j}^{r}$ to the number of vertices in a sphere $S_{r}$.
\end{proof}

An important quantity to relate volume growth to spectral theory is the exponential growth rate. For a set $W\subseteq V$, define
\begin{align*}
    \mathrm{vol}(W)=\sum_{v\in W}d_{v}.
\end{align*}
Let $E_{W}$ be the set of edges that have both vertices in $W$. We find
\begin{align*}
    \mathrm{vol}(W)=2\#E_{W}+\#\partial W
\end{align*}
which can be interpreted as twice the number of edges within $W$ and once the number of edges leaving $W$.
The \emph{exponential growth rate} is defined as
$$\mu=\limsup_{r \to \infty} \frac{1}{r}\log \mathrm{vol}(B_r).$$
For connected  graphs $\mu$ does not depend on the choice of the center $o$ of the balls.

In order to relate the results above to $\mu$ we need the following lemma.
\begin{lemma}Let $G$ be a tessellation of a surface of finite genus. Then,
\begin{align*}
 \mu=\limsup_{r \to \infty} \frac{1}{r}\log \#B_r
 =\limsup_{r \to \infty} \frac{1}{r}\log \#S_r.
\end{align*}
\end{lemma}
\begin{proof}
Let $g$ be the genus of the surface. Using the fact that every face is included in two edges and every face includes at least three vertices we derive from Poincar\'e's formula for finite $W\subseteq V$
\begin{align*}
    \# E_{W}  \leq C \# W
\end{align*}
with $C=\max\{2g + 1, 3\}$, (for details see \cite[Lemma 6.2.]{BGK1}). We derive for $r\ge1$
\begin{align*}
    \# B_{r}\leq \mathrm{vol}(B_{r})=2\#E_{B_{r}}+\#\partial B_{r}\le 2\#E_{B_{r+1}}\leq 2C\# B_{r+1}.
\end{align*}
This yields the first inequality. To see the second equality let $R_{k}$ be a subsequence realizing the $\limsup$. Let $0\leq r_{k}\leq R_k$ such that $\#S_{r_{k}}\ge\# S_{j}$, $0\leq j\leq R_k$. Then,
\begin{align*}
  \frac{1}{R_{k}}\log \#B_{R_{k}}= \frac{1}{R_{k}}\log\Big( \#S_{r_{k}}\sum_{j=0}^{R_{k}}\frac{\# S_{j}}{\# S_{r_{k}}}\Big)\leq \frac{1}{r_{k}}\log \#S_{r_{k}}+\frac{1}{R_{k}}\log R_{k}.
\end{align*}
Hence,  the assertion follows by taking limits.
\end{proof}

\subsection{Isoperimetric inequalities}\label{s:isoperimetric}

The isoperimetric problem goes back to a tale about Dido of Carthage. Here, we consider such a problem in the graph case and want to minimize the ratio of the boundary area and the volume of finite sets. The ratio is called the isoperimetric constants and it has many applications in mathematics and computer science, as e.g. applications to mixing times of Markov chains which is relevant in simulation, \cite{LPW,MT}.

A rough synopsis of this section is that non-negative curvature implies zero isoperimetric constant, while negative curvature implies positive isoperimetric constant.

In the negative curvature case, we further discuss an explicit formula for $(p,q)$-regular tessellations, give lower bounds as results of upper bounds on the curvature and give a criterion for positivity by asymptotically negative curvature. As a corollary we show that this implies Gromov hyperbolicity of the tessellation. Finally, we discuss an isoperimetric constant at infinity.

Let us be more  precise. We let
\begin{align*}
    \partial W =\{(v,w)\in W\times (V\setminus W)\mid v\sim w\}
\end{align*}
for a finite set $W\subseteq V$. This set can be interpreted as the edges leaving $W$ and the area of the boundary will be number of elements in $\partial W$. Furthermore, as in the previous section we let the volume of a finite set $W$ be given by
\begin{align*}
    \mathrm{vol}(W)=\sum_{v\in W}d_{v}.
\end{align*}

We define the \emph{isoperimetric constant} to be
\begin{align*}
    \al=\inf_{W\subset V \mbox{{\scriptsize finite}}}\frac{\#\partial W}{\mathrm{vol}(W)}.
\end{align*}
By the formula  $\mathrm{vol}(W)=2\#E_{W}+\#\partial W$ we conclude
$$0\leq \al< 1.$$

One can also define the volume by the number of vertices of the set and define the isoperimetric constant accordingly. However,  measuring volume by counting edges as above proves to be more effective for spectral estimates as the estimate does not become trivial for unbounded vertex degree; compare  the results of \cite{Do} and \cite{DKe}. Another choice for the area of the boundary is to involve so called intrinsic metrics, see \cite{BKW}. Although the corresponding isoperimetric constant works well for spectral estimate, the constant is hard to estimates combinatorially in the tessellation case.

We first start with a theorem which is a corollary of Theorem~\ref{t:grow_nonneg_curvature}.

\begin{thm}\label{t:iso_nonneg_curvature} Let $G$ be a tessellation embedded in a surface of finite genus such that $\ka\ge0$. Then,
\begin{align*}
\al=0
\end{align*}
\end{thm}
\begin{proof}First of all notice that $\ka\leq0$ implies $d_{v}\leq 6$ by Proposition~\ref{p:ka_vs_deg}.  Thus, $\#\partial B_{r}\leq 6\#S_{r}$.
As $\mathrm{vol}(B_{r})\ge \#B_{r}\ge\#S_{r}+\#S_{r-1}$, we derive
\begin{align*}
    \al\leq \frac{\#\partial B_{r}} {\mathrm{vol}(B_{r})}\leq \frac{6 \# S_{r}}{\# S_{r}+\#S_{r-1}}.
\end{align*}
We set $a=\frac{\al/6}{1-\al/6}$ and observe that $a>1$ whenever $\al>0$.
We conclude by iteration
$a^{r}\leq \# S_{r}$
which stands in contradiction to  $\#B_{r}\leq C r^{2}$ of Theorem~\ref{t:grow_nonneg_curvature}.
\end{proof}

We consider now graphs with negative curvature. First, we mention that the isoperimetric constant can be calculated explicitly for $(p,q)$-regular tessellations. These formulas were independently obtained by \cite{HaJL} and \cite{HiShi} by quite different techniques.

\begin{thm}\cite{HaJL,HiShi}
  Let $G$ be a $(p,q)$-regular tessellation such that $\ka_{C}=\frac{1}{p}-\frac{1}{2} + \frac{1}{q} \le0 $. Then,
  $$ \al = \frac{q-2}{q} \sqrt{1 - \frac{4}{(p-2)(q-2)}}. $$
\end{thm}

While \cite{HiShi} give a rather explicit construction of the minimizing sets, the authors of \cite{HaJL} use duality. We give a rough sketch of the proof.
\begin{proof}[Idea of proof \cite{HaJL}] Recall that $E_{W}$ are the edges with both vertices in $W$ and consider
\begin{align*}
    \beta(G)&= \inf_{W\subseteq V \mbox{\scriptsize{ finite}}}\frac{\#W}{\#E_W},\\
    \delta(G)&= \inf_{W\subseteq V \mbox{\scriptsize{ finite}}} \frac{\#W}{\#E_W+\#\partial W}.
\end{align*}
It can be observed that
\begin{align*}
    \beta(G)=\frac{2}{q(1-\al)}\quad\mbox{and}\quad\delta(G) =\frac{2}{q(1+   \al)}.
\end{align*}
For the proof of the theorem one needs to show that the minimizing sets for $\al,\beta$ and $\delta$  have to grow whenever $\al>0$. Knowing this the major step in the proof is to show the relation
\begin{align*}
    \beta(G)+\delta(G^{*})=1,
\end{align*}
where $G^{*}$ is the dual tessellation.
Resolving this formula with the equations for $\beta$ and $\delta$ above yields the statement.
\end{proof}

We now look at the non-regular case, where we have a uniform upper bound on the curvature. In this case we get explicit estimates on $\al$. Although, these estimates are not sharp for planar tessellations, they are sharp if one considers sequences of tessellations where the face degrees grow to infinity. This is discussed in \cite{KP}.

Whenever the face degree is bounded by some $p$ and the vertex degree is bounded by some $q$ the following constant $C_{p,q}\ge1$ will enter the estimate of the isoperimetric constant below
\begin{equation*}
C_{p,q}:= \begin{cases} 1 &:\; \text{if $p= \infty$}, \\
1 + \frac{2}{p-2} &:\; \text{if $p < \infty$ and $q=\infty$}, \\
\Big(1 + \frac{2}{p-2}\Big) \Big(1 + \frac{2}{(p-2)(q-2)-2}\Big) &:\; \text{if
$p,q < \infty$}.
\end{cases}
\end{equation*}

\begin{thm} \label{t:curv:cheegest} \cite[Theorem~1]{KP}
  Let $G$ be a planar tessellation such that
  $\deg(f) \le p$ for all $f \in F$ and $d_{v} \le q$ for all $ v \in V$  with $p,q\in[3,\infty]$.
  Assume $\ka<0$ and let $K := \inf_{v \in V} -\frac{1}{d_{v}}\ka(v)$. Then
    $$ \al \ge 2 C_{p,q}K. $$
\end{thm}

A key insight for the proof is a formula which is attributed in \cite{BP1} to Harm Derksen. It is an immediate consequence of the Gau\ss-Bonnet theorem and direct calculation. It can be interpreted that the formula for the curvature of a simply connected finite set has the same form as the curvature of a vertex. Here, we call a set $W\subseteq V$ \emph{simply connected} if $W$ and $V\setminus W$ are connected.

\begin{lemma}\label{l:Derksen}\cite[Proposition~2.1.]{BP1} Let $W\subseteq V$ be a finite simply connected subset of a planar tessellation. Then,
\begin{align*}
    \sum_{v\in W}\ka(v)=1-\frac{\#\partial W}{2}+\sum_{f\in F, f\cap W\neq\emptyset, f\cap(V\setminus W)\neq \emptyset}\frac{\#(f\cap W)}{\deg(f)}.
\end{align*}
\end{lemma}
From this lemma, one immediately derives the estimate
\begin{align*}
    \frac{\#\partial W}{\mathrm{vol}(W)}\ge\frac{  -2\sum_{v\in W}\ka(v) }{\mathrm{vol}(W)}\ge2K.
\end{align*}
To obtain the statement of Theorem~\ref{t:curv:cheegest} on still has to squeeze in $C_{p,q}\ge1$. This needs some further rather subtle considerations.

Next, we turn to a result of Woess \cite{Woe}. It states that if the asymptotic curvature  defined as
\begin{align*}
    \ov\ka=\lim_{n\to\infty}\sup_{W\subseteq V, \# W\geq n}\frac{1}{\#W}\sum_{v\in W}\ka(v)
\end{align*}
is negative, then the isoperimetric constant is positive.
The limit exists or is $-\infty$  as the sequence is monotone decreasing.

\begin{thm}\label{t:iso_neg_av_curvature}\cite[Theorem~A]{Woe} Let $G$ be a planar  tessellation such that $\ov\ka<0$. Then, $\al>0$.
\end{thm}

Again the result can be proven by the formula in Lemma~\ref{l:Derksen} above which is however not stated explicitly in \cite{Woe}. Preceding the result above, Dodziuk proved in \cite{Do} that planar graphs with $d_{v}\ge7$, $v\in
V$,  satisfy $\al>0$. In particular, this assumption
implies $\ka<0$ by Proposition~\ref{p:ka_vs_deg}.

By considerations of Oh \cite{Oh}, we see that planar tessellations satisfying $\ov \ka<0$ are Gromov hyperbolic provided there is an upper bound on the face degree. A metric space is called \emph{Gromov hyperbolic} if there is $\delta>0$ such that  every geodesic triangle $T$ is $\delta$-thin,  i.e., any
side of $T$ is contained in the $\delta$-neighborhood of the union of the other two sides. For graphs we consider  the combinatorial graph distance $d$ as a metric.

\begin{cor}\cite{Oh} Let $G$ be a planar  tessellation such that $\ov\ka<0$ and assume there is an upper bound on the face degree. Then, $(G,d)$ is Gromov hyperbolic.
\end{cor}
\begin{proof} By \cite[Theorem~6]{Oh} we see that a planar graph is Gromov hyperbolic if a constant called $\ka(G)$, which is defined in \cite{Oh}, is positive. By \cite[Theorem 1]{Oh} this constant $\ka(G)$ is positive if and only if our isoperimetric constant  $\al$ is positive. This equivalence holds true for all planar graphs whose faces are homeomorphic to a closed disk, an assumption which is satisfied for tessellations.
\end{proof}

To complement this result we remark that by a result of Cao \cite{Cao} one can check that Gromov hyperbolicity of a planar tessellation with bounded vertex degree implies $\al>0$. For details see \cite[Proof of Theorem 3.8]{KPP}.

We end this section by discussing an asymptotic isoperimetric constant called the isoperimetric constant at infinity. This constant is discussed in \cite{Fu2,Mo91,K1}. We define the \emph{isoperimetric constant at infinity} to be
\begin{align*}
        \al_{\infty}=\sup_{K\subset V \mbox{{\scriptsize finite}}}\inf_{W\subset V\setminus K \mbox{{ \scriptsize finite}}}\frac{\#\partial W}{\mathrm{vol}(W)}.
\end{align*}
Clearly, $\al\leq\al_{\infty}$. Furthermore, it is discussed in \cite[Section~6.3]{BHK} that $\al>0$ if and only if $\al_{\infty}>0$. The arguments employed there use basic spectral theory.

Accordingly, we define an upper bound on the \emph{curvature at infinity} by
\begin{align*}
    \ka_{\infty}=\inf_{K\subset V \mbox{{\scriptsize finite}}}\sup_{v\in V\setminus K}\ka(v).
\end{align*}

For  the curvature at infinity and the isoperimetric constant at infinity can be related as a consequence of \cite{K1} and Proposition~\ref{p:ka_vs_deg}.

\begin{thm}\label{t:iso_neg_av_curvature}\cite[Proposition~6]{K1} Let $G$ be a planar  tessellation such that $\ka_{\infty}=-\infty$. Then, $\al_{\infty}=1$.
\end{thm}


\section{Spectral theory}\label{s:spectrum}
In this section we study the spectral theory of the combinatorial Laplacian on tessellations. Many of the results are based on the geometric insights gathered in the previous section.

Since the definition of the combinatorial Laplacian does not depend on the tessellating structure, we introduce it on general graphs. It shall be stressed that although the Laplace operator considered here has the same quadratic form as the normalized Laplacian introduced in the previous chapter of Jost  it is defined on a different Hilbert space. This arises from a different notion of volume. Indeed, there are two canonical ways to measure volume on graphs, either by counting vertices or edges. The volume considered by Jost is the function $\mathrm{vol}$ used in the previous section and is associated to counting edges. Here, we use the counting measure which means volume is determined by counting vertices.

Both viewpoints have there merits. The normalized Laplacian considered by Jost based on the edge volume captures perfectly the metric features which come from the combinatorial graph distance, see e.g. \cite[Section~3.2.2.]{K4}. Moreover, this operator has the advantage that it is always bounded which avoids various technicalities. On the other hand, phenomena which are related to unbounded geometry are much better captured by the combinatorial Laplacian we study here.
In summary both Laplacians arise from natural geometric considerations and each is suitable for the study of certain phenomena. Specifically, we study here discreteness of spectrum in the case of uniformly decreasing curvature together with eigenvalue asymptotics and decay properties of eigenfunctions. We also study spectral bounds implied by geometric data and consider unique continuation properties of eigenfunctions.

\subsection{The combinatorial Laplacian}
In this section we introduce the combinatorial Laplacian on graphs. This operator has many applications in various fields of mathematics. We introduce the operator for  general graphs since the restriction to tessellations yields no additional basic information or properties.

Let $(V,E)$ be a graph. Then, the combinatorial Laplacian $\Delta$ acts on functions $\ph:V\to\R$ as
\begin{align*}
\Delta \ph(v)=\sum_{w\sim v}(\ph(w)-\ph(v)),\qquad v\in V.
\end{align*}
In order to study  spectral theory,  we  restrict $\Delta$ to a space of functions with more structure. We refer to \cite{Weidmann} for a background on operator theory. Let $\ell^{2}(V)$ be the space of square summable real valued functions, i.e.,
\begin{align*}
    \ell^{2}(V)=\{\ph:V\to\R\mid \sum_{v\in V} |\ph(v)|^{2}<\infty\}.
\end{align*}
The space $\ell^{2}(V)$ equipped with the scalar product
$$\langle\ph,\psi\rangle=\sum_{v\in V}\ph(v)\psi(v)$$
is a Hilbert space. We denote the corresponding norm by $\|\cdot\|$.

One can check easily that $\Delta$ is a bounded operator on $\ell^{2}(V)$ if and only if
\begin{align*}
    \sup_{v\in V}d_{v}<\infty.
\end{align*}
In the case, where $\Delta$ is unbounded we have to restrict $\Delta $ to a dense subspace of $\ell^{2}(V)$ to define a selfadjoint operator.
It was first  shown  in \cite[Theorem~1.3.1.]{Woj1} that $\Delta$ is a selfadjoint  on
\begin{align*}
D(\Delta)=\{\ph\in\ell^{2}(V)\mid& \Delta\ph\in\ell^{2}(V)\}.
\end{align*}
From now on we refer to $\Delta$ restricted to $D(\Delta)$ as the \emph{combinatorial Laplacian}.

It can be shown
that the functions of finite support $C_{c}(V)$ are dense in $D(\Delta)$ with respect to the graph norm. (Note that the  graph norm is a functional analytic quantity referring to the norm $\|\ph\|+\|\Delta\ph\|$ on $D(\Delta)$.) Since
\begin{align*}
\langle     \Delta  \ph,\ph\rangle=-\frac{1}{2}\sum_{v,w\in V,v\sim w}(\ph(v)-\ph(w))^{2},
\end{align*}
the operator $-\Delta $ is positive. In what follows we study the spectral theory of $-\Delta$. Clearly, $D(\Delta)=D(-\Delta)$.

We denote the spectrum of $-\Delta$ by
\begin{align*}
    \si(-\Delta)=\{z\in\C\mid -\Delta-z\mathrm{Id} \mbox{ has no bounded inverse}\},
\end{align*}
where $\mathrm{Id}$ is the identity operator on $\ell^{2}(V)$.  Since $-\Delta$ is selfadjoint and positive, we have by general theory
\begin{align*}
    \si(-\Delta)\subseteq [0,\infty).
\end{align*}
In the case  $\sup_{v}d_{v}\leq D$, we even have $$\si(-\Delta)\subseteq [0,2D].$$

For finite graphs the spectrum only consists of eigenvalues of $\Delta$. If the graph is infinite the space $\ell^{2}(V)$ is infinite dimensional and, therefore, there might be values $\lm\in \si(\Delta)$ which do not allow for an $\ell^{2}(V)$ eigenfunction. However, by a criterion of Weyl one still has approximate eigenfunctions.

We denote the \emph{bottom of the spectrum} of $-\Delta$ by $\lm_{0}=\lm_{0}(-\Delta)=\min\si(-\Delta)$. We have by the Rayleigh-Ritz characterization and the density of $C_{c}(V)$ in $\ell^{2}(V)$
\begin{align*}
    \lm_{0}(-\Delta)=\inf_{ \ph\in C_{c}(V),\,\|\ph\|=1}\langle     (-\Delta) \ph,\ph \rangle.
\end{align*}
In the case where $\lm_{0}>0$ one  says that $-\Delta$ has a \emph{spectral gap}.

For finite graphs the constant functions are in $\ell^{2}(V)$ and are, therefore,  eigenfunctions to the eigenvalue $0$. For infinite graphs the constant functions are never in $\ell^{2}(V)$. However, in some cases it is still possible that the constant functions can be approximated with respect to graph norm by functions in $\ell^{2}(V)$ in which case there is no spectral gap.

It is said that $-\Delta$ has \emph{pure discrete spectrum} if $\si(-\Delta)$ consists only of discrete eigenvalues of finite multiplicity. This implies that they only accumulate at infinity. Clearly, this can only happen if $\Delta $ is unbounded. In the case of pure discrete spectrum we denote the eigenvalues of $-\Delta$ by $\lm_{n}$, $n\ge0$, in increasing order counted with multiplicity.

The part of the spectrum which are no discrete eigenvalues of finite multiplicity is called the \emph{essential spectrum} and is denoted by $\si_{\mathrm{ess}}(-\Delta)$. Furthermore, we let
\begin{align*}
 \lm_{0}^{\mathrm{ess}}(-\Delta)=\min \si_{\mathrm{ess}}(-\Delta).
\end{align*}

\subsection{Bottom of the spectrum}
 Let us recall some well known results from discrete spectral geometry to estimate the bottom of the spectrum.
These results are classically proven for general graphs and the normalized Laplacian which we denote here by $\ow\Delta$ . We relate the bottom of spectra $\lm_{0}(-\ow \Delta)$ to $\lm_{0}(-\Delta)$ and $\lm_{0}^{\mathrm{ess}}(-\ow \Delta)$ to $\lm_{0}^{\mathrm{ess}}(-\Delta)$ as it was shown by very elementary arguments in \cite{K1}
\begin{align*}
    m\lm_{0}(-\ow\Delta)\leq \lm_{0}(-\Delta)\leq \lm_{0}^{\mathrm{ess}}(-\Delta)\leq M_{\infty}\lm_{0}^{\mathrm{ess}}(-\ow\Delta)
\end{align*}
with $m=\min_{v\in V}d_{v}$
and  $M_{\infty}=\inf_{K\subseteq V,\mbox{\scriptsize{ finite}}}\sup_{v\in V\setminus K}d_{v}$ whenever $M_{\infty}<\infty$.

Recall the definition of the volume growth rate
$$\mu=\limsup_{r \to \infty} \frac{1}{r}\log \mathrm{vol}(B_r)$$
from Section~\ref{s:volume} and
the isoperimetric constant
\begin{align*}
    \al=\inf_{W\subset V \mbox{{\scriptsize finite}}}\frac{\#\partial W}{\mathrm{vol}(W)}.
\end{align*}
from Section~\ref{s:isoperimetric}. Then, by results of Fujiwara \cite{Fu1,Fu2}
\begin{align*}
    m\big(1-\sqrt{1-\al^{2}}\big)\leq\lm_{0}(-\Delta)\leq \lm_{0}^{\mathrm{ess}}(-\Delta)\leq M_{\infty}\Big(1-\frac{2e^{\mu/2}}{e^{\mu}+1}\Big).
\end{align*}
The lower bound was preceded by a result of Dodziuk/Kendall \cite{DKe} and is found in a similar form in the work of Mohar \cite{Mo88}. The upper bound is preceded by results of Dodziuk/Karp \cite{DKa} and Ohno/ Urakawa \cite{OU}.

We first use the upper bound in the case of non-negative curvature. We conclude by the estimate above and Theorem~\ref{t:grow_nonneg_curvature}.

\begin{cor} Let $G$ be a tessellation embedded  in a surface of finite genus such that $\ka\ge0$. Then, $    \lm_{0}(-\Delta)=\lm_{0}^{\mathrm{ess}}(-\Delta)=0$.
\end{cor}

Furthermore, we use Theorem~\ref{t:iso_neg_av_curvature} to derive the following.
\begin{thm} Let $G$ be a planar  tessellation such that $\ov\ka<0$. Then, $\lm_{0}>0$.
\end{thm}

Finally, we combine both estimates above with Theorem~\ref{t:growest} and Theorem~\ref{t:iso_nonneg_curvature}.

\begin{thm}\label{t:spectral_estimates}
  Let $G$ be a planar tessellation such that
  $\deg(f) \le p$ for all $f \in F$ and $d_{v} \le q$ for all $ v \in V$  with $p,q\in[3,\infty]$.
  Assume $\ka<0$ and let $K= \inf_{v \in V} -\frac{1}{d_{v}}\ka(v)$. Then,
   \begin{align*}
 2mK^{2}\le m(1-\sqrt{1-4C_{p,q}^{2}K^{2}}) \le \lm_{0}(-\Delta),
\end{align*}
where $C_{p,q}$ is defined in Section~\ref{s:isoperimetric}. If additionally $p,q<\infty$, then
\begin{align*}
    \lm_{0}^{\mathrm{ess}}(-\Delta)\leq M_{\infty}\Big(1-\frac{2 x_{p,q}^{1/2}}{x_{p,q}+1}\Big),
\end{align*}
where $x_{p,q}$ is the largest real zero of $g_{p,q}$ defined in Section~\ref{s:volume}.
\end{thm}

The lower bound above can be considered as a discrete analogue to a
theorem of  McKean \cite{McK}. McKean proved for a $n$-dimensional
complete Riemannian manifold $M$ with upper sectional curvature
bound $-k$ that the bottom of the spectrum of the  Laplace-Beltrami
$-\Delta_{M}$ satisfies
\begin{align*}
(n-1)^{2}k/4  \le   \lm_{0}(-\Delta_{M}).
\end{align*}
It shall be noted that by Theorem~\ref{t:Higuchi} the assumption $\ka<0$  implies $K>0$ and, therefore, $\lm_{0}(-\Delta)>0$ in this case.

\subsection{Discrete spectrum,  eigenvalue asymptotics and decay of eigenfunctions}
We now turn to a characterization of pure discrete spectrum. The theorem below is  an analogue of a theorem of Donnelly/Li \cite{DL} from Riemannian geometry. Recall the definition of the upper asymptotic curvature bound
\begin{align*}
        \ka_{\infty}=\sup_{K\subset V\,\mathrm{finite}}\inf_{v\in V\setminus K}\ka(v)
\end{align*}

\begin{thm}\label{t:discretespectrum}\cite[Theorem~3]{K1} Let $G$ be a planar tessellation. Then the spectrum of $-\Delta$ is purely discrete if and only if
$\ka_{\infty}=-\infty$.
\end{thm}
\begin{proof}[Idea of proof] If the spectrum of $-\Delta$ is purely discrete, then, by a criterion attributed to Perrson  in the continuous case, $\langle (-\Delta) \ph_{n},\ph_{n}\rangle\to\infty$ for every normalized sequence $(\ph_{n})$ in $\ell^{2}(V)$ which converges weakly to zero. For a sequence $(v_{n})$ of vertices and the delta functions $\delta_{v_{n}}$, we have
\begin{align*}
    \langle (-\Delta)\delta_{v_{n}},\delta_{v_{n}}\rangle=d_{v_{n}}\leq-2\ka(v_{n}),
\end{align*}
where the last inequality follows from  Proposition~\ref{p:ka_vs_deg}. This implies $\ka_{\infty}=-\infty$.\\
On the other hand, assume $\ka_{\infty}=-\infty$. By Theorem~\ref{t:curv:cheegest} we infer that outside of large enough finite sets the isoperimetric constant is uniformly positive. Moreover, by an argument as in Theorem~\ref{t:spectral_estimates} the bottom of the spectrum of the operator restricted to functions supported outside of larger and larger finite sets converges to $\infty$. This, however, is equivalent to pure discrete spectrum, as the restricted operators are finite rank perturbations of the original operator.
\end{proof}

In \cite{Fu2} Fujiwara proved a similar statement as the theorem above for the
normalized Laplacian on trees, namely that spectrum is discrete except for the point $1$, where the discrete eigenvalues accumulate. Furthermore,  Wojciechowski \cite{Woj1} showed  discreteness of the
spectrum of  $-\Delta$ on general graphs in terms of a different curvature quantity sometimes  referred to as a mean curvature.

We observe the following standard fact. Whenever, the spectrum of $-\Delta$ is purely discrete, then $\lm_{0}(-\Delta)>0$: If $0$ was in the spectrum, then it must be a discrete eigenvalue. However, the only eigenfunctions $\ph$ to $0$ which have finite energy, (i.e., $\sum_{v\sim w}(\ph(v)-\ph(w))^{2}<\infty$) are the constant functions which are not in $\ell^2(V)$ in the case of infinite graphs).

In the case of discrete spectrum we next present asymptotics for eigenvalue $\lm_{n}$ of $-\Delta$, i.e.,
\begin{align*}
    \si(-\Delta)=\{0<\lm_{0}\le\lm_{1}\leq\ldots\leq\lm_{n}\leq\ldots\}
\end{align*}

For the case $\ka_{\infty}=-\infty$, it follows from Proposition~\ref{p:ka_vs_deg} that the vertices can be ordered  $V=\{v_{n}\}$ such that
\begin{align*}
    d_{v_{n}}\leq d_{v_{n+1}},\qquad n\ge0.
\end{align*}

\begin{thm}\label{t:eigenvalues}\cite[Theorem 1.6.]{BGK2} Let $G$ be a planar tessellation such that $\ka_{\infty}=-\infty$. Then,
\begin{align*}
     d_{v_n}-2\sqrt{d_{v_n}} \lesssim \lm_{n}\lesssim d_{v_n}+2\sqrt{d_{v_n}},
\end{align*}
that is
\begin{align*}
    \lim_{n\to\infty} \frac{\lm_{n}}{d_{v_n}}=1
\end{align*}
and
\begin{align*}
-2\leq \liminf_{n\to\infty}  \frac{\lm_{n}-d_{v_n}}{\sqrt{d_{v_n}}}\leq \limsup_{n\to\infty}  \frac{\lm_{n}-d_{v_n}}{\sqrt{d_{v_n}}}\leq 2.
\end{align*}
\end{thm}
\begin{proof}[Idea of proof]
The proof  in \cite{BGK2} is divided into two steps. First one shows that every tessellation whose curvature decreases to $-\infty$ allows for a spanning tree such that the combinatorial Laplacian on the tessellation and the tree are bounded perturbations of each other. Equivalently this means that there is a number $N$ such that from each vertex at most $N$ adjacent edges are canceled to obtain the spanning tree. By the Min-Max-Principle both operators share the same eigenvalue asymptotics. \\
Secondly, one shows the corresponding eigenvalue  for the operator on the tree. The proof uses isoperimetric techniques and again a version of the  Min-Max-Principle.
\end{proof}

In the case of planar tessellations with constant face degree one can  show bounds with an even more geometric flavor.
Assume all faces have degree $p$. Then, the internal angle of every face $f$ is given by
\begin{align*}
    \beta(f)=\beta(p)=2\pi\frac{(p-2)}{p}.
\end{align*}

\begin{cor}\label{c:eigenvalues}\cite[Corollary 1.8.]{BGK2} Let $G$ be a planar tessellation such that $\ka_{\infty}=-\infty$. Suppose the face degree is constantly $p\ge3$ outside of some finite set. Then, for large $n$
\begin{align*}
    -\frac{2\pi\ka(v_{n})}{\beta(p)}-2\sqrt{-\frac{2\pi\ka(v_{n})}{\beta(p)}} \lesssim \lm_{n}\lesssim  -\frac{2\pi\ka(v_{n})}{\beta(p)}+2\sqrt{-\frac{2\pi\ka(v_{n})}{\beta(p)}},
\end{align*}
that is
\begin{align*}
-\frac{2\pi\ka(v_{n})}{{\lm}_{n}}\to \beta(p),
\end{align*}
and
\begin{align*}
-\frac{\beta(p)}{2}\leq& \liminf_{n\to\infty}  \frac{\sqrt{-2\pi\ka(v_{n})}}{\lm_{n}+2\pi\ka(v_{n})/\beta(p)} \leq \limsup_{n\to\infty}  \frac{\sqrt{-2\pi\ka(v_{n})}}{\lm_{n}+2\pi\ka(v_{n})/\beta(p)}\leq \frac{\beta(p)}{2}.
\end{align*}
\end{cor}

The eigenvalue asymptotics in the theorem and the corollary above present a case where phenomena in  the discrete and continuous world drift apart. In particular, for Riemannian manifolds one expects upper bounds of eigenvalues $\lm_{k}$  by some constant multiplied by $k^{2}$, see \cite{KLP}.

Next, we come to the decay of eigenfunctions. It turns out that eigenfunctions decay exponentially in an $\ell^{2}$ sense.

\begin{thm}\label{t:expdecay}\cite{KL3} If $\ka_{\infty}=-\infty$ and $\ph_{n}\in D(\Delta)$, $n\ge0$, are eigenfunctions, i.e.,
$$    \Delta\ph_{n}=\lm_{n}\ph_{n},$$
then, for any $\beta<e^{-1}$ and $o\in V$,
$$    |\ka|^{\frac{1}{2}}e^{\beta d(o,\cdot)}\ph_{n}\in\ell^{2}(V),$$
where $d(\cdot,\cdot)$ is the natural graph distance.
\end{thm}

The proof is based on ideas of Agmon for
Schr\"odinger operators in $\R^{n}$. The somewhat curious constant
$e^{-1}$ comes in via an optimization that is necessary by the
non-locality of the graph Laplacian in contrast to the strongly
local Laplace operator on $\R^{n}$.

\subsection{Unique continuation of eigenfunctions}

Under the phenomena unique continuation one understands that a function which is zero outside of a compact set must also be zero on the compact set.  Such  unique continuation properties for eigenfunctions hold in great generality  for local elliptic operators.
Often very strong quantitative statements, called Carleman estimates, can be proven which all have the basic corollary that there are no eigenfunctions supported on a compact subset of the space. However, for graphs such
eigenfunctions can be produced  rather easily. This was observed by many authors, see e.g.
\cite{AW,BK,DLMSY,KLS}. Below we will also discuss examples. The purpose of this section is to discuss that
non-positive corner curvature excludes such  a phenomena, i.e., there are no compactly eigenfunctions.

Klassert/Lenz/Peyerimhoff/Stollmann \cite{KLPS} proved a
unique continuation result for tessellations with non-positive corner curvature. This result  was later generalized to planar graphs in \cite{K2} with a different proof.

\begin{thm}\cite[Theorem~4]{KLPS}
Let $G$ be a planar tessellation such that $\ka_{C}\le0$. Then, there are no eigenfunctions of compact support.
\end{thm}
\begin{proof}[Idea of proof] The proof given in \cite{K2} works by a polar decomposition of the Laplacian into $\Delta=E^{\top}+D+E$, where $D$ is the combinatorial Laplacian on the spheres with Dirichlet boundary conditions and $E$ is the operator with matrix elements $E(v,w)=1$ if $v\in S_{r+1}$, $w\in S_{r}$, $v\sim w$ for all $r\ge0$ and zero otherwise. Hence, the operator $E$ encodes the edges reaching a sphere from the previous one. Denote the restriction of $E$ and $D$ to the functions supported on $S_{r}$ by $E_{r}$ and $D_{r}$. For an eigenfunction $\ph$ let $\ph_{r}$ be the restriction of $\ph$ to $S_{r}$. Then,  the eigenvalue equation $\Delta\ph=\lm\ph$ reads in polar decomposition as
\begin{align*}
    E_{r}\ph_{r}+(D-\lm)\ph_{r+1}+E^{\top}_{r+1}\ph_{r+2}=0,\qquad r\ge0.
\end{align*}
Now, by subtle geometric arguments it is shown by induction over $r$ that $E_{r}$ is injective. Therefore, if $\ph$ vanishes outside of $B_{r}$, the above equation reads as $    E_{r}\ph_{r}=0$ which implies $\ph_{r}=0$ by injectivity of $E_{r}$.
\end{proof}

Let us discuss, that the statement of the theorem fails, if one only assumes $\ka\leq0$.

\begin{eg}
Consider the Kagome lattice of Example~\ref{e:Kagome}. Pick a hexagon and denote the vertices of the adjacent triangles which are not in the hexagon
by $v_{1},\ldots, v_{6}$. Now let $\ph:V\to\{-1,1\}$ be supported on $v_{1},\ldots, v_{6}$ such that $\ph(v_{j})=(-1)^{j}$, $j=1,\ldots,6$. Then it follows that $\Delta\ph=6\ph$ and $\ph$ is therefore an eigenfunction of compact support.
The same idea works if one considers the other examples of Example~\ref{e:Kagome} when we replace the hexagon with a $2p$-gon with $p>3$. In this case we even have $\ka<0$.
\end{eg}

A question that arises is whether it is sufficient that the corner curvature is non positive outside of a finite set in order to have at most finitely many compactly supported eigenfunctions. Still this is not the case as the example in Figure~\ref{f:cplysupportedef} shows.

\begin{figure}[!h]
\scalebox{0.3}{\includegraphics{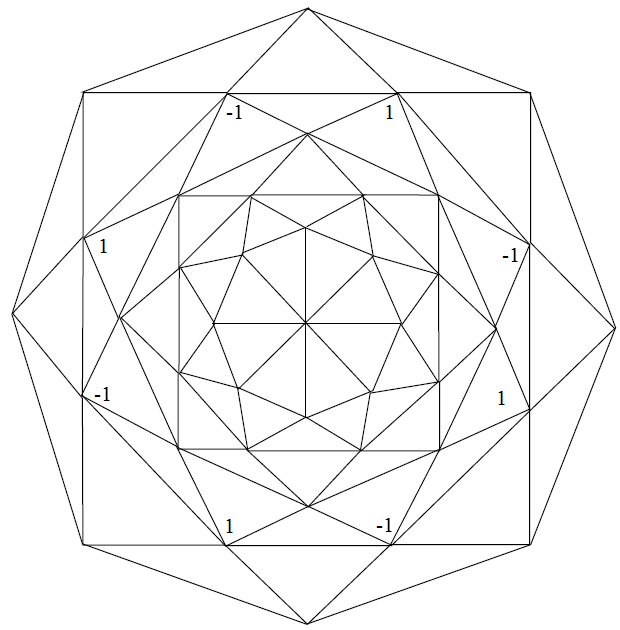}}
\caption{\label{f:example} The first six distance spheres of  a tessellation that admits infinitely many linearly independent compactly supported eigenfunctions and has $\ka_{C}\leq 0$ outside of $B_{2}$}\label{f:cplysupportedef}
\end{figure}

However, if the curvature is sufficiently large outside of a finite set a unique continuation result can be proven.

\begin{thm}\label{t:main2}\cite[Theorem 1.9.]{BGK2}  Let $G$ be a planar graph. Assume $\ka_{\infty}=-\infty$.
Then, outside of a finite set there are no eigenfunctions of compact support of $\Delta$. In particular, there are at most finitely many linearly independent eigenfunctions of finite support.
\end{thm}

\subsection{$\ell^{p}$ spectrum}
In this section we study the spectral theory of  $-\Delta $ as an operator on the Banach spaces $\ell^{p}(V)$, $p\in[1,\infty]$. Denote the restriction of $\Delta$ to
$$D(\Delta_{p})=\{f\in \ell^{p}(V)\mid \Delta f\in\ell^{p}(V)\}$$
by $\Delta_{p}$ and let $\Delta_{\infty}=\Delta_{1}^{*}$.
Clearly, $D(\Delta_{p})=D(-\Delta_{p})$.

A  question asked by Simon
\cite{Si0} and affirmatively answered by Hempel/Voigt \cite{HV1} for Schr\"odinger
operators is whether the spectrum depends on the underlying Banach
space. Sturm, \cite{Stu1}, addressed this question in the setting of uniformly elliptic operators on manifolds. As a special  case,  he considers consequences of curvature bounds.
For the tessellation case the following result is found in \cite{BHK}.

\begin{thm}\cite[Theorem~7.1.,~7.2.~and~7.3.]{BHK} Let $G$ be a planar tessellation.\\
(a) If $\ka\ge0$, then $\sigma(-\Delta_{2})=\sigma(-\Delta_{p})$ for   $p\in [1,\infty]$.\\
(b) If $-K\leq \ka<0$, then $\lm_{0}(-\Delta_{2})\neq\lm_{0}(-\Delta_{1})$, i.e., $\sigma(-\Delta_{2})\neq\sigma(-\Delta_{1})$.\\
\hspace{1cm}(c) If $\ka_{\infty}=-\infty$, then       $\sigma(-\Delta_{2})=\sigma(-\Delta_{p})$ for all  $p\in (1,\infty)$.
\end{thm}

\section{Extensions to more general graphs}
In this final section we discuss how the results of the previous section can be generalized to planar graphs and more general polygonal complexes.
\subsection{Curvature on planar graphs}\label{s:planar}
First we address general planar graphs as it was studied in \cite{K2}. The first step is to extend the notion of curvature. Secondly, we show that non-positive already implies that the graph looks locally like a tessellation. Such graphs allow for a suitable embedding into a tessellation which is used to extend the results for tessellations to general planar graphs.

Let $(V,E)$ be a planar graph which is embedded locally compactly into a surface homeomorphic to $\R^{2}$ which gives rise to faces $F$ as above.
To define curvature on planar graphs, we have to extend the definitions of degrees of faces and vertices. In order to do so, we introduce the degree of a corner. As above the corners $C(G)$ are the pairs $(v,f)\in V\times F$ such that $v\in f$. For a corner $(v,f)\in C(G)$ we define the \emph{degree} $|(v,f)|$ by the minimal number of times the vertex $v$ is met by a boundary walk of $f$. Roughly speaking the degree of the corner $(v,f)$ is the number of times $f$ touches $v$.

This gives rise to a definition of the \emph{degree} of  $v\in V$ and $f\in F$ by
\begin{align*}
    d_{v}=\sum_{(v,g)\in C(G)}|(v,g)|\quad\mbox{and}\quad     \deg(f)=\sum_{(w,f)\in C(G)}|(w,f)|.
\end{align*}
In a tessellation the degree of a corner is always one. So, these definitions indeed extend the ones of tessellations.

A face $f$ is called \emph{unbounded} if $\deg(f)=\infty$. Moreover, a planar graph is {simple} if and only if $\deg(f)\ge 3$ for all $f\in F$.

We define the \emph{corner curvature } $\ka_{C}:C(G)\to\R$ by
\begin{align*}
    \ka_{C}(v,f)=\frac{1}{d_{v}}-\frac{1}{2}+\frac{1}{\deg(f)}
\end{align*}
and the \emph{vertex curvature} by $\ka:V\to\R$ by
\begin{align*}
\ka (v)=\sum_{(v,f)\in C(G)}|(v,f)|\ka_{C}(v,f).
\end{align*}
These definitions are consistent with the definition of $\ka_{C}$ and $\ka$ on tessellations. Furthermore, they allow to show a Gau\ss-Bonnet formula, \cite[Proposition~1]{K2}.

It turns out that non-negative curvature already has strong consequences on the structure of the graph. To this end, we look at a generalization of tessellations. We call a face a \emph{polygon} if it is homeomorphic to an open  disc and we call it an \emph{infinigon} if it is homeomorphic to the upper half space in $\R^{2}$.  For example the faces of a tessellation are all polygons by (T3) and the faces of a tree are infinigons.

We call a planar graph \emph{locally tessellating} if it satisfies the following properties:
\begin{itemize}
  \item [(T1)] Every edge is contained in two faces.
  \item [(T2*)] Two faces are either disjoint or intersect in a vertex or in a path of  edges. If this path consists of  more than one edge then both faces are unbounded.
  \item [(T3*)] Every face is a polygon or an infinigon.
\end{itemize}
The assumption (T1) is the same as in the definition of tessellations.

Tessellations, trees as well as hybrids of both of these are examples of locally tessellating graphs. As mentioned above non-positive curvature on planar graphs implies a graph is almost a tessellation.

\begin{thm}\cite[Theorem~1]{K2} Let $G$ be a connected planar graph. If  $\ka_{C}\leq0$ or if $G$ is simple with $\ka\le0$ then $G$ is locally tessellating and infinite.
\end{thm}
\begin{proof}[Idea of proof]
To prove this theorem one isolates finite subgraphs with simply closed boundary of the tessellation on which some of the assumptions (T1), (T2*), (T3*) fail. Then one copies this subgraph finitely many times and pastes the copies along their boundary paths to be finally embedded into the  unit sphere.
Now, the Gau\ss-Bonnet theorem implies that there must be some positive curvature.
\end{proof}

By \cite[Theorem~2]{K2} locally tessellating graphs can be embedded into tessellations in a way that approximates the original curvature arbitrarily close.

\begin{thm}\label{t:embedding} {Let $G$ be a  locally tessellating graph that satisfies $\ka\leq0$. Let $W\subset V$ be a finite and simply connected set and $\eps\in(0,1/1806)$. Then, there is a tessellation $G'$ which is a supergraph of $G$ such that the following properties hold.
\begin{itemize}
\item[(a)]  The embedding of $G$ into the supergraph $G'$ is a graph isomorphism of the subgraphs $G_W$ and $G_{W}'$ and the  embedding is an isometry on $W$.
\item[(b)]  If $\ka_C\leq 0 $, then $\ka_C'\le\ka_C+\eps$ and if  $\ka\leq 0 $, then $\ka'\leq \ka+\eps$.
\end{itemize}}
\end{thm}

By doing so, one can carry over results from tessellations to locally tessellating graphs and by  Theorem~\ref{t:embedding} above to planar graphs in the case of non-positive curvature.

Among the geometric applications are the following:
\begin{itemize}
\item Absence of cut locus for non-positive corner curvature \cite[Theorem~3]{K2}.
\item Bounds for the growth of distance balls for negative vertex curvature     \cite[Theorem~5]{K2}.
\item Positivity and bounds for the isoperimetric constant constant for negative vertex curvature   \cite[Theorem~6]{K2}.
 \item  Gromov hyperbolicity for  negative corner curvature     \cite[Theorem~7]{K2}.
\end{itemize}

There are also applications in spectral theory.
\begin{itemize}
  \item The geometric bounds yield bounds on the bottom of the spectrum.
  \item  Uniformly decreasing curvature is equivalent to purely discrete spectrum \cite[Theorem~8]{K2}.
  \item The same eigenvalue asymptotics as in Theorem~\ref{t:eigenvalues} \cite{BGK2}, and the same decay of eigenfunctions as in Theorem~\ref{t:expdecay} \cite{KL3}.
  \item Unique continuation results for eigenfunctions \cite[Theorem~9]{K2}.
\end{itemize}

\subsection{Sectional curvature for polygonal complexes}

To end this chaper we  discuss generalizations to non planar graphs for which we can define a notion of sectional curvature. In \cite{Wi}  such a program was undertaken to study questions in group theory. In \cite{KPP} another somewhat more restrictive approach was taken to define curvature for polygonal complexes which however allows to prove various results in geometry and spectral theory.

Let us be more precise. In \cite{KPP} a \emph{polygonal complex} $G=(V,E,F)$ is said to have\emph{ planar substructures} if there is a system $\mathcal{A}$ of subcomplexes called the \emph{apartments} that satisfy the following axioms:
\begin{itemize}
  \item [(PCPS1)] For every two faces there is an apartment which contains both of them.
  \item [(PCPS2)] The apartments are convex, that is every geodesic of faces, which starts and ends in an apartment, stays completely in the apartment.
  \item [(PCPS3)] The apartments are planar or spherical tessellations.
\end{itemize}
An important example of such polygonal complexes are two dimensional buildings.

We define the  degree $\deg(f)$ of a face $f\in F$ as in the case of tessellation by the number of edges included in $f$. Moreover, we denote the  degree of a vertex $v$ with respect to an apartment $\Sigma$ by $d_{v}^{(\Sigma)}$. The corners of an apartment are denote by $C_{\Sigma}(G)$.
For an apartment $\Sigma=(V_{\Sigma},E_{\Sigma}, F_{\Sigma})$, we define the \emph{sectional corner curvature}  $\ka_{C}:C_{\Sigma}(G)\to\R$ with respect to $\Sigma$ by
\begin{align*}
    \ka_{C}^{(\Sigma)}(v,f)=\frac{1}{d_{v}^{\Sigma}}-\frac{1}{2}+\frac{1}{\deg(f)}
\end{align*}
and the \emph{sectional face curvature} by $\ka^{(\Sigma)}_{F}:F_{\Sigma}\to\R$ by
\begin{align*}
\ka^{(\Sigma)}_{F} (f)=\sum_{(v,f)\in C_{\Sigma}(G)}\ka_{C}^{(\Sigma)}(v,f) =1- \frac{\deg(f)}{2}+\sum_{v\in V_{\Sigma}, v\in f}\frac{1}{d_{v}^{\Sigma}}.
\end{align*}

With this definition one can prove various of the results of Section~\ref{s:geometry} and Section~\ref{s:spectrum}. Here, we refrain from stating the results precisely but only mention them and refer to \cite{KPP} for details:
\begin{itemize}
\item Finiteness and infiniteness depending on the sign of the curvature, \cite[Theorem~3.13.]{KPP}
\item Absence of cut locus for non-positive sectional corner curvature \cite[Theorem 3.1.]{KPP}
\item Positivity and bounds for an isoperimetric constant constant for negative sectional  face curvature,    \cite[Theorem~3.8.~and~3.11.]{KPP}
 \item  Gromov hyperbolicity for negative sectional corner curvature,     \cite[Theorem~3.6.]{KPP}
\end{itemize}

To study spectral theory one considers the combinatorial Laplacian on functions defined on faces. This is due to the fact that the geometric assumptions for the polygonal complexes are made with respect to the faces. For a function $\ph:F\to\R$, we define
\begin{align*}
    \Delta \ph(f)=\sum_{g\in F,\ g\sim f}(\ph(g)-\ph(f)),
\end{align*}
where $g\sim f$ means that $f$ and $g$ share an edge. Restricting $\Delta$ to
\begin{align*}
    D(\Delta)=\{\ph\in\ell^{2}(F)\mid \Delta\ph\in\ell^{2}(F)\}
\end{align*}
gives rise to a selfadjoint operator. For this operator one can prove:
\begin{itemize}
\item Discreteness of spectrum and eigenvalue asymptotics under the assumption of uniform decreasing corner curvature, \cite[Theorem~4.1.]{KPP}
\item Unique continuation of eigenfunctions, \cite[Theorem~4.3.]{KPP}.
\end{itemize}




\textbf{Acknowledgement}.
MK enjoyed the hospitality of C.I.R.M. and acknowledges the financial support of the German Science Foundation (DFG).

\bibliographystyle{alpha}
\newcommand{\etalchar}[1]{$^{#1}$}

\end{document}